\documentclass{amsart}

\usepackage{mathtools}
\usepackage{amsmath}
\usepackage{amssymb}
\usepackage{amsthm}
\usepackage{enumitem}
\usepackage{lastpage}
\usepackage{fancyhdr}
\usepackage{accents}
\usepackage{scrextend}
\usepackage{tikz-cd}
\usepackage{graphicx}
\graphicspath{ {./images/} }

\usepackage[bb=dsserif]{mathalpha}
\usepackage{bm}

\usepackage{xr-hyper}
\usepackage{hyperref}

\newtheorem{theorem}{Theorem}[section]
\newtheorem{lemma}[theorem]{Lemma}
\newtheorem{corollary}[theorem]{Corollary}
\newtheorem{proposition}[theorem]{Proposition}

\theoremstyle{definition}
\newtheorem{definition}[theorem]{Definition}

\newtheorem{remark}[theorem]{Remark}
\numberwithin{equation}{section}

\newcommand{\Z}{\mathbb{Z}}

\newcommand{\HP}{\mathbb{H}P}

\newcommand{\Spinc}{\text{Spin$^c$ }}
\newcommand{\Spinh}{\text{Spin$^h$ }}
\newcommand{\spinc}{\text{spin$^c$ }}
\newcommand{\spinh}{\text{spin$^h$ }}
\newcommand{\Spin}{\text{Spin}}
\newcommand{\Sp}{\text{Sp}}
\newcommand{\Ahath}{\hat{\mathcal{A}}^h}

\DeclareMathOperator{\rk}{rank}
\DeclareMathOperator{\id}{id}

\DeclareMathOperator{\Sq}{Sq}

\begin{document}

\title{The Structure of the Spin$^h$ Bordism Spectrum}

\author{Keith Mills}
\address{University of Maryland, College Park}
\curraddr{}
\email{kmills96@umd.edu}


\date{December 8, 2023}


\begin{abstract}
Spin$^h$ manifolds are the quaternionic analogue to \spinc manifolds. We compute the \spinh bordism groups at the prime $2$ by proving a structure theorem for the cohomology of the \spinh bordism spectrum $\mathrm{MSpin^h}$ as a module over the mod 2 Steenrod algebra. This provides a 2-local splitting of $\mathrm{MSpin^h}$ as a wedge sum of familiar spectra. We also compute the decomposition of $H^*(\mathrm{MSpin^h};\mathbb{Z}/2\mathbb{Z})$ explicitly in degrees up through 30 via a counting process.
\end{abstract}

\maketitle

\section{Introduction} \label{sec:intro}

Spin$^h$ manifolds are the quaternionic analogue to \spinc manifolds. $\Spin^h(n)$ is a central extension of $SO(n) \times \Sp(1)$ by $\Z_2 = \mathbb{Z}/2\mathbb{Z}$, and a \spinh structure on an oriented $n$-manifold is a lifting of the principal frame bundle from $SO(n)$ to $\Spin^h(n)$. We aim to compute the \spinh bordism groups $\Omega_{*}^{spin^h}.$ As explained in (a section of) Jiahao Hu's thesis \cite{Hu}, there is a \spinh bordism spectrum $\mathrm{MSpin^h}$, so computing the bordism groups is equivalent to determining the homotopy groups of this spectrum.

\Spinh manifolds have been the subject of recent research of various flavors. In \cite{albanesemilivojevic2021spinh} it is shown that there is a notion of $\Spin^k$ for any integer $k \geq 1$, and \spinh is the case of this for $k=3$. More geometrically, Hu shows in his thesis \cite{Hu} that \spinh manifolds carry Dirac operators with indices in symplectic $K$-theory and that there is a symplectic version of the Atiyah-Singer index theorem. They also show up as a special case of the work of Freed and Hopkins on quantum field theories and Thom spectra in \cite{freedhopkins2021refl}. One can find a more detailed exposition on \spinh manifolds and consequences of their study in Lawson's article \cite{lawson2023spinh}.

It turns out that with $2$ inverted, the \spinh bordism groups are determined by the spin bordism groups. Precisely, one has $\Omega_*^{spin^h}[\frac{1}{2}] \cong \Omega_*^{spin} \otimes_{\Z} H_*(\HP^\infty;\Z[\frac{1}{2}])$ (\cite{Hu}, A.0.2), so it suffices to concern ourselves with computing the $2$-primary part of the \spinh bordism groups. We also mention that Hu provides the \spinh bordism groups in degrees 0 through 6 as $\Z, 0, 0, 0, \Z \oplus \Z, \Z_2 \oplus \Z_2, \Z_2 \oplus \Z_2$; this is stated without proof but is a consequence of our main result.

The standard approach for computing the $2$-primary part of $\pi_*(\mathrm{MSpin^h})$ is to use the mod 2 Adams spectral sequence, which takes as input $H^*(\mathrm{MSpin^h};\Z_2)$ as a module over the mod 2 Steenrod algebra $A$. Consequently, the computation of this $A$-module becomes the central focus of our work.

For the classical computations for spin and \spinc bordism one leverages the fact that the corresponding spectra are ring spectra. Unfortunately, unlike $\mathrm{MSpin}$ and $\mathrm{MSpin}^c$, $\mathrm{MSpin^h}$ is not a ring spectrum since the product of two \spinh manifolds may not be spin$^h$. However, the product of a spin and \spinh manifold is spin$^h$, which does buy us some extra structure: $\mathrm{MSpin^h}$ is an $\mathrm{MSpin}$-module spectrum. $\mathrm{MSpin}$-module spectra are studied by Stolz in \cite{stolz1992scpsc} and \cite{stolz1994splitting}.

A key implication of Stolz's work is that $H^*\mathrm{MSpin^h}$ (coefficient groups $\Z_2$ are omitted) is an ``extended" $A$-module, meaning that it is determined by the $A(1)$-module structure of a certain submodule $\overline{H^*\mathrm{MSpin^h}}$ of $H^*\mathrm{MSpin^h}$, where $A(1)$ is the subalgebra of $A$ generated by $\Sq^1$ and $\Sq^2$. Explicitly, it is determined via the $A$-module isomorphism $H^*\mathrm{MSpin^h} \cong A \otimes_{A(1)} \overline{H^*\mathrm{MSpin^h}}$. This is analogous to the situation for spin bordism, where it is the case that $H^*\mathrm{MSpin} \cong A \otimes_{A(1)} N$ for a certain $A(1)$-module $N$, which is computed by Anderson, Brown, and Peterson in \cite{abp1967spin}.

It turns out that $\mathrm{MSpin^h} \simeq \mathrm{MSpin} \wedge \Sigma^{-3}MSO_3$ as $\mathrm{MSpin}$-module spectra, which has the consequence that $\overline{H^*\mathrm{MSpin^h}} \cong N \otimes_{\Z_2} H^*(\Sigma^{-3}MSO_3)$ as $A(1)$-modules. As such, the determination of $\overline{H^*\mathrm{MSpin^h}}$ relies on understanding the $A(1)$-module structure of these two factors; the first, $N$, is very well-known, and it so happens that the second is easy to compute.

Our main result is a combination of Theorem \ref{thm:mspinhstructure} and Corollary \ref{cor:mspinhhtpy}, which is a structure theorem for $\mathrm{MSpin^h}$ and states that 2-locally it splits up as a generalized wedge sum of familiar spectra with known homotopy groups. While we are unable to provide a closed-form expression for these homotopy groups in each degree (a consequence of the fact that there is not one for $\mathrm{MSpin}$) we are able to provide a counting process that allows one to compute any particular homotopy group of interest.

To state the main result, fix the following notation. Let $I$ be the augmentation ideal of $A(1)$, which corresponds\footnote{When we say that a spectrum $Y$ ``corresponds" to an $A(1)$-module $M$, we mean that $\overline{H^*Y} = M$; the definition of $\overline{H^*Y}$ in general will be given in section \ref{sec:prelims}. For $\mathrm{ksp}$ and $H\Z_2$ this correspondence will be clear, and it is possible to show that $\Sigma\widetilde{\mathrm{ko}}$ is the cofiber of the map $\mathrm{ko} \to H\Z_2$ given by the nontrivial element of $H^0\mathrm{ko}$.} to a spectrum $E$ that has the same homotopy groups as the connective cover of the real $K$-theory spectrum $\mathrm{ko}$ except for a degree shift. We let $\widetilde{\mathrm{ko}}$ be $\Sigma^{-1}E$ (corresponding to $\Sigma^{-1} I$) so that its homotopy groups are the same as those for $\mathrm{ko}$. Let $K$ be the quotient of $A(1)$ by $(A(1)\Sq^1 + A(1)\Sq^{2,1,2})$, which corresponds to the spectrum $\mathrm{ksp} = \Sigma^{-4}\mathrm{ko}\langle 4 \rangle$, where $\mathrm{ko}\langle 4 \rangle$ is the 3-connected cover of $\mathrm{ko}$. Then our main result is

\begin{theorem}\label{thm:mainthm}
    The $A(1)$-module $\overline{H^*\mathrm{MSpin^h}}$ is a direct sum of suspensions of $A(1)$, $I$, and $K$. Hence the Adams spectral sequence implies that $\mathrm{MSpin^h}$ has the 2-local homotopy type of a generalized wedge sum of copies of the Eilenberg-MacLane spectrum $H\Z_2$, $\widetilde{\mathrm{ko}}$, and $\mathrm{ksp}$.
\end{theorem}

We also compute this decomposition in degrees $\leq 30$, the first few terms of which are
\[
\mathrm{MSpin^h} \simeq_{2} \mathrm{ksp} \vee \Sigma^4\widetilde{\mathrm{ko}} \vee 2(\Sigma^8 \mathrm{ksp}) \vee \Sigma^9 H\Z_2 \vee \Sigma^{10}H\Z_2 \vee 3(\Sigma^{12}\widetilde{\mathrm{ko}}) \vee \Sigma^{14}H\Z_2 \vee \cdots
\]
\noindent (an integer $n$ in front of a particular wedge summand means $n$ copies of that summand).

The paper is structured as follows. In section \ref{sec:prelims} we provide a short exposition on \spinh manifolds and describe the cohomology of the classifying space for principal \Spinh bundles, and then discuss the results of Stolz's work on $\mathrm{MSpin}$-module spectra that are relevant for our computations.

In section \ref{sec:spinh_structure} we prove that as an $\mathrm{MSpin}$-module spectrum, $\mathrm{MSpin^h}$ is equivalent to $\mathrm{MSpin} \wedge \Sigma^{-3}MSO_3$, and then describe the structure results for the $A(1)$-modules $N = \overline{H^*\mathrm{MSpin}}$ and $H^*MSO_3$; each of these has a decomposition into smaller well-studied $A(1)$-modules, so following this we compute the pairwise tensor products necessary to understand $\overline{H^*\mathrm{MSpin^h}} \cong N \otimes \Sigma^{-3}H^*MSO_3$. From here, the main structure theorem \ref{thm:mainthm} follows.

In section \ref{sec:counting} we describe the counting process that provides the explicit decomposition mentioned above, which is essentially the combination of locating each module in the decompositions of $N$ and $H^*MSO_3$ and then computing the resulting tensor product, recording where each piece ends up. We end in section \ref{sec:more} with a short discussion of a map $\Ahath: \mathrm{MSpin^h} \to \mathrm{ksp}$ which generalizes the spin orientation $\mathcal{A}: \mathrm{MSpin} \to \mathrm{ko}$, and from our decomposition it follows that this map is projection onto the bottom wedge summand of $\mathrm{MSpin^h}$.

\section*{Acknowledgments} \label{sec:acknowledgments}

I would like to thank Jonathan Rosenberg for his assistance and encouragement at all stages in the preparation of this paper. I would also like to thank Jiahao Hu for several conversations on the topic of \spinh manifolds and for suggesting a sketch of Lemma \ref{lem:mspinh_smashsplit2}, and the reviewer for improvements on an earlier version of the paper.

\section{Preliminaries} \label{sec:prelims}

\subsection{\Spinh Manifolds} \label{subsec:spinh_mflds}

In this subsection we review the requisite background and information on \spinh manifolds. A more detailed account can be found in \cite{albanesemilivojevic2021spinh}. For convenience, we assume all manifolds under discussion are smooth, compact, and oriented.

Recall that $\Spin(n)$ is the universal (double) cover of $SO(n)$. Consider the ``diagonal" $\Z_2$ subgroup of $\Spin(n) \times \Sp(1)$ generated by $(-1,-1)$. Taking the quotient by this subgroup gives the group $\Spin^h(n) \coloneqq \Spin(n) \times_{\Z_2} \Sp(1)$. One can verify that $\Spin^h(n)$ sits in an exact sequence of groups
\[
1 \to \Z_2 \to \Spin^h(n) \to SO(n) \times SO(3) \to 1.
\]

The $h$ is used to emphasize the role of the quaternions, just as the $c$ in \Spinc is used to highlight the role of the complex numbers. The homomorphism $\Spin^h(n) \to SO(n) \times SO(3)$ is a double covering, and composing with projections we have natural maps $\Spin^h(n) \to SO(n)$ and $\Spin^h(n) \to SO(3)$.

We are ready to state our definitions of interest.

\begin{definition} \label{def:spinh_vect}
A \textit{\spinh structure} on a principal $SO(n)$-bundle $P$ is a principal $SO(3)$-bundle $E$ together with a principal $\Spin^h(n)$-bundle $Q$ and a double covering $Q \to P \times E$ that is equivariant with respect to $\Spin^h(n) \to SO(n) \times SO(3)$.
\end{definition}

\begin{definition} \label{def:spinh_mfld_bordism}
A manifold $M$ is a \spinh manifold if its tangent bundle admits a spin$^h$ structure. \Spinh manifolds with boundary and bordism of \spinh manifolds are defined in the usual way.
\end{definition}

It is shown in \cite{albanesemilivojevic2021spinh} that the first obstruction to the existence of a \spinh structure on an oriented manifold $M$ is the fifth integral Stiefel-Whitney class $W_5$, but that this is not the only obstruction, in contrast with spin and \spinc structures on manifolds where $w_2$ and $W_3$ are the only obstructions (resp.). This is a consequence of the fact that the homotopy fiber of the map $BSpin^h(n) \to BSO(n)$ is not an Eilenberg-MacLane space, where $BSpin^h(n)$ denotes the classifying space for \spinh vector bundles of rank $n$.

Also in contrast to the situation for spin and \spinc manifolds is the fact that the product of two \spinh manifolds need not be spin$^h$: for example, the product $W \times W$ of the Wu manifold $W = SU(3)/SO(3)$ is not \spinh since $W_5(W \times W) \neq 0$. We do, however, have the partial result (\cite{albanesemilivojevic2021spinh}, 3.6) that if $M$ is spin, then $M \times N$ is \spinh if and only if $N$ is spin$^h$; also, if a product $M \times N$ is \spinh then both $M$ and $N$ are spin$^h$.


We now record the mod 2 cohomology of the stable classifying space $BSpin^h$, computed by Hu in section 2 of \cite{Hu}. Recall that the mod $2$ cohomology of $BSO$ is generated by the universal Stiefel-Whitney classes $w_i$ for $i \geq 2$, and denote by $\nu_k$ the universal Wu class in degree $k$.

\begin{theorem} \label{thm:Hu_HBSpinh}
$H^*(BSpin^h; \Z_2)$ is generated additively by the products of the universal Stiefel-Whitney classes $w_i$ for $i \geq 2$ and $i \neq 2^{r+1}+1$ for $r \geq 1$. As an algebra, $H^*(BSpin^h;\Z_2)$ is naturally isomorphic to the quotient $H^*(BSO;\Z_2) / (\Sq^1 \nu_{2^{r+1}})_{r \geq 1}.$
\end{theorem}

Note that the classes $w_{2^{r+1}+1}$ are not necessarily zero in $H^*(BSpin^h;\Z_2)$, but rather are decomposable; for example, one may compute that $w_9 = w_2w_7 + w_3w_6$.

\subsection{The spectrum $\mathrm{MSpin^h}$} \label{subsec:mspinh} 

As mentioned in the introduction, while $\mathrm{MSpin^h}$ is not a ring spectrum it is an $\mathrm{MSpin}$-module spectrum. $\mathrm{MSpin}$-module spectra are studied by Stolz in \cite{stolz1992scpsc} and \cite{stolz1994splitting}. Here we record some of his discussion and results on such spectra that will be relevant for our purposes. In what follows, suppress homology and cohomology coefficient groups by assuming they are $\Z_2$, and assume any tensor products without subscripts are over $\Z_2$.

Denote by $A$ the mod $2$ Steenrod algebra and by $A(1)$ the subalgebra of $A$ generated by $\Sq^1$ and $\Sq^2$. Let $A_*$ and $A(1)_*$ be their respective duals. Recall that $A_*$ is a polynomial algebra $\Z_2[\xi_1, \xi_2, \ldots]$ on the ``Milnor generators" $\xi_i$ of degree $2^i - 1$. Typically one works with a different set of generators given by $\zeta_i = c(\xi_i)$, where $c$ is the conjugation involution on $A_*$. In these generators one has $A(1)_* = \Z_2[\zeta_1, \zeta_2]/(\zeta_1^4, \zeta_2^2)$.

Let $Y$ be an $\mathrm{MSpin}$-module spectrum. Then $H_*Y$ is a comodule over the dual Steenrod algebra $A_*$, and information about the homotopy groups of $Y$ at the prime 2 can be obtained from the mod 2 Adams spectral sequence, which takes as input $H_*Y$ as an $A_*$-comodule. Stolz shows the following:

First, $H_*Y$ is a module over $R = H_*\mathrm{ko} = A_* \square_{A(1)_*} \Z_2 = \Z_2[\zeta_1^4, \zeta_2^2, \zeta_3, \zeta_4,\ldots]$. Here, $\mathrm{ko}$ denotes the connective real $K$-theory spectrum, and the cotensor product $A_* \square_{A(1)_*} M$ is defined for any (left) $A(1)_*$-comodule $M$ by the exact sequence below, where $\psi$ denotes the $A(1)_*$-comodule structure maps for both $A_*$ and $M$:
\[
0 \to A_* \square_{A(1)_*} M \to A_* \otimes M \xrightarrow{\psi \otimes 1 - 1 \otimes \psi} A_* \otimes A(1)_* \otimes M.
\]

Second, if $H_*Y$ is bounded below and locally finite, then there is a natural isomorphism $\Phi_Y: H_*Y \to A_* \square_{A(1)_*} \overline{H_*Y}$ compatible with the $A_*$-comodule and $R$-module structures, where $\overline{H_*Y}$ is the ``$R$-indecomposable quotient" of $H_*Y$ given by $\overline{H_*Y} \coloneqq \Z_2 \otimes_{R} H_*Y$ (\cite{stolz1992scpsc}, 5.5). One refers to this by saying that the $A_*$-comodule $H_*Y$ is an ``extended" $A(1)_*$-comodule, extended from $\overline{H_*Y}$. The map $\Phi_Y$ fits into the following commutative diagram (where it is a result of Stolz that the map $1 \otimes \pi$ factors through $A_* \square_{A(1)_*} \overline{H_*Y}$):

\begin{center}
	\begin{tikzcd}
		H_*Y \arrow{dr}{\Phi_Y} \arrow{r}{\text{coaction}} &A_* \otimes H_*Y  \arrow{d} \arrow{r}{1 \otimes \pi} &A_* \otimes \overline{H_*Y} \\
		& A_* \square_{A(1)_*} \overline{H_*Y} \arrow[hookrightarrow]{ur}{\text{incl.}}
	\end{tikzcd}
\end{center}

Now, since the homology of $Y = \mathrm{MSpin^h}$ is bounded below and locally finite, the above result applies. This means that instead of directly computing the $A_*$-comodule structure on $H_*\mathrm{MSpin^h}$, we can instead compute the $A(1)_*$-comodule $\overline{H_*\mathrm{MSpin^h}}$ and then form the required cotensor product.

Let us also explicitly describe the dual situation. There is a natural map $\Phi^Y: A \otimes_{A(1)} \overline{H^*Y} \to H^*Y$ that is an $A$-module and $H^*\mathrm{ko}$-comodule isomorphism as well as a $H^*\mathrm{MSpin}$-module map. Here, $\overline{H^*Y} = \Z_2 \square_{H^* \mathrm{ko}} H^*Y$, and $\Phi^Y$ fits into the following commutative diagram, where $\text{ev}$ refers to the $A$-module structure on $H^*Y$:

\begin{center}
	\begin{tikzcd}
	A \otimes \overline{H^*Y} \arrow{dr}{\text{quotient}} \arrow{r}{1 \otimes \text{incl.}} &A \otimes H^*Y \arrow{r}{\text{ev}} &H^*Y \\
		&A \otimes_{A(1)} \overline{H^*Y} \arrow{u} \arrow{ur}{\Phi^Y}
	\end{tikzcd}
\end{center}

\section{Structure Results on $\mathrm{MSpin^h}$} \label{sec:spinh_structure}

We wish to describe the structure of $\mathrm{MSpin^h}$. As a starting point one has the following result.

\begin{lemma} \label{lem:mspinh_smashsplit2}
There is a 2-local homotopy equivalence of spectra $\mathrm{MSpin} \wedge \Sigma^{-3} MSO_3 \to \mathrm{MSpin^h}$.
\end{lemma}
\begin{proof}

First, note that Definition \ref{def:spinh_vect} is equivalent to saying that a $Spin^h(n)$ vector bundle is a pair $(E_1,E_2)$ consisting of a principal $SO(n)$-bundle $E_1$ and $SO(3)$-bundle $E_2$ such that $w_2(E_1) = w_2(E_2)$. Then given a principal $Spin(n)$-bundle $E$ and $SO(3)$-bundle $F$, the pair $(E \oplus F, F)$ is a $Spin^h(n+3)$-bundle since $w_2(E \oplus F) = w_2(E) + w_1(E)w_1(F) + w_2(F) = w_2(F)$, where the first two terms vanish since $E$ is spin.

This construction gives a map $f: BSpin(n) \times BSO_3 \to BSpin^h(n+3)$. We claim that this map induces an isomorphism on mod 2 cohomology as $n \to \infty$. By the Thom isomorphism theorem, the claim implies that the corresponding map on the Thom space level $\mathrm{MSpin}(n) \wedge MSO_3 \to \mathrm{MSpin^h}(n+3)$ is an isomorphism on mod 2 cohomology as $n \to \infty$. Thus $f$ provides a map $\mathrm{MSpin} \wedge MSO_3 \to \Sigma^3 \mathrm{MSpin^h}$ that induces an $A$-module isomorphism on cohomology, hence by the Adams spectral sequence a 2-local homotopy equivalence (strictly speaking, the Adams spectral sequence gives a 2-completed homotopy equivalence, but we can conclude that $f$ gives a 2-local equivalence since we know the structure of these spectra with 2 inverted).

Now let us prove the claim. In the way of notation, denote elements of $H^*BSO_{n+3}$ by $w_i$s, $H^*BSO_n$ by $\tilde{w}_i$s, and $H^*BSO_3$ by $w_i'$s. Now note that there is a pullback diagram

\begin{center}
	\begin{tikzcd}
	BSpin(n) \times BSO_3 \arrow{r}{f} \arrow{d}{\alpha} & BSpin^h(n+3) \arrow{d}{\beta} \\
	BSO_n \times BSO_3 \arrow{r}{g} & BSO_{n+3} \times BSO_3
	\end{tikzcd}
\end{center}

\noindent where $\alpha$ is obtained by taking the fiber of $BSO_n \times BSO_3 \to K(\Z_2,2)$ corresponding to $\tilde{w}_2 \in H^2(BSO_n \times BSO_3)$ (or, put more simply, $\alpha$ is the map $Bp \times \id$, where $p$ is the double covering of $SO_n$ by $Spin(n)$) and $\beta$ is obtained by taking the fiber of $BSO_{n+3} \times BSO_3 \to K(\Z_2,2)$ corresponding to $w_2 + w_2' \in H^2(BSO_{n+3}\times BSO_3)$, and $g$ is the map induced by sending a pair of bundles $(E,F)$ to $(E \oplus F, F)$. This diagram is a pullback diagram because the corresponding map on groups is a pullback diagram, and the classifying space functor preserves small limits.

Further, a straightforward computation gives
\begin{align*}
		&g^*(w_2) = \tilde{w}_2 + w_2' & &g^*(w_2') = w_2' \\
		&g^*(w_3) = \tilde{w}_3 + w_3' & &g^*(w_3') = w_3' \\
		&g^*(w_i) = \tilde{w}_i + \tilde{w}_{i-2}w_2' + \tilde{w}_{i-3}w_3', i > 3
\end{align*}

\noindent so it is apparent that $g^*$ is surjective with kernel concentrated in degrees $\geq n+1$. Consequently, as $n \to \infty$ we obtain that $g^*$ is an isomorphism and hence the same is true for $f^*$.
\end{proof}

This gives the stronger result
\begin{corollary}\label{cor:mspinh_smashsplit}
$\mathrm{MSpin^h} \simeq \mathrm{MSpin} \wedge \Sigma^{-3}MSO_3$ as $\mathrm{MSpin}$-module spectra.
\end{corollary}
\begin{proof}
In the construction $(E,F) \mapsto (E \oplus F, F)$, summing a spin bundle onto either side only involves summing with the first component, hence $f$ is a map of $\mathrm{MSpin}$-module spectra. One also concludes that $f$ is an equivalence with $2$ inverted, since with $2$ inverted $Spin^h(n) \simeq Spin(n) \times Sp(1) = Spin(n) \times SO(3)$ and $f$ is certainly a rational isomorphism. Since all torsion is $2$-torsion, $f$ gives an isomorphism on integral cohomology, and by the Whitehead theorem is a homotopy equivalence.
\end{proof}

By work discussed in section \ref{sec:prelims} (explicitly, this is a consequence of Corollary 5.5 discussed in section 6 of \cite{stolz1992scpsc}), one has that $\overline{H^*\mathrm{MSpin^h}} \cong \overline{H^*\mathrm{MSpin}}\otimes_{\Z_2} H^*MSO_3$ so to compute the $E_2$-term of the Adams spectral sequence for $\mathrm{MSpin^h}$ we need only compute the $A(1)$-module structure of $\overline{H^*\mathrm{MSpin}}\otimes_{\Z_2} H^*MSO_3$, where the action of $A(1)$ on this tensor product is given by the Cartan formula.

The $A(1)$-module structure of each factor above is well-understood and is given by the following lemma. Let $I$ be the augmentation ideal of $A(1)$, $J$ the ``joker" $A(1)/(A(1)\Sq^{1,2})$, and $K = A(1)/(A(1)\Sq^1 + A(1)\Sq^{2,1,2})$ (note that $K = \overline{H^*\mathrm{ksp}} = \Sigma^{-4}\overline{H^*\mathrm{ko}\langle 4 \rangle}$). Images of these $A(1)$-modules are depicted below, where a straight line from a basis element $x$ to a basis element $y$ depicts that $\Sq^1(x) = y$, and respectively for curved lines and $\Sq^2$. We remark that our conventions have the lowest nonzero element in $I$, $J$, or $K$ in degree $1, 0$ or $0$ (resp.).

\begin{center}
    \includegraphics[scale=0.75]{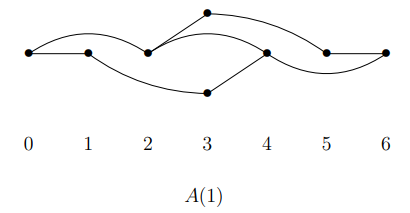}\includegraphics[scale=0.75]{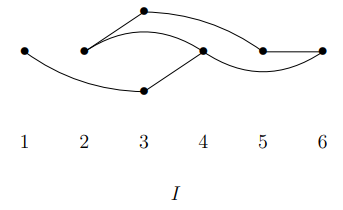}

    \includegraphics[scale=0.85]{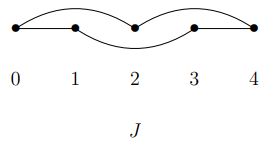}\includegraphics[scale=0.85]{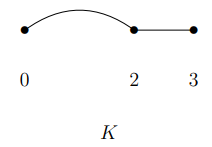}
\end{center}

\begin{proposition} \label{prop:mspin_mso}
$\overline{H^*\mathrm{MSpin}}$ is isomorphic to a direct sum of suspensions of $A(1)$, $A(1)/I \cong \Z_2$, and $J$. $H^*MSO_3$ is isomorphic to a direct sum of suspensions of $A(1)$, $I$, and $K$.
\end{proposition}

The first statement is the content of \cite{abp1967spin}. The second is easily computed by hand, and its computation is carried out in detail by Freed and Hopkins in Appendix D of \cite{freedhopkins2021refl} (see D.4, D.5).

Proposition \ref{prop:mspin_mso} implies that as an $A(1)$-module, $\overline{H^*\mathrm{MSpin^h}}$ is a direct sum of suspensions of each of the pairwise tensor products above. These tensor products (over $\Z_2$) are easily computable as follows.
\begin{lemma} \label{lem:piecestensor}
\begin{enumerate}
	\item For any connected $A(1)$-module $M$, $A(1) \otimes M$ is a free $A(1)$-module of rank $\rk_{\Z_2}M$.
	\item $J \otimes K \cong \Sigma I \oplus A(1)$.
	\item $J \otimes \Sigma^{-1} I \cong \Sigma^2 K \oplus A(1) \oplus \Sigma A(1) \oplus \Sigma^2 A(1) \oplus \Sigma^3 A(1)$.
\end{enumerate}
\end{lemma}
\begin{proof}
For (1), recall that the K\"unneth formula for Margolis homology gives $H^*(L \otimes M; z) \cong H^*(L;z) \otimes H^*(M;z)$ where $z$ is $Q_0 = \Sq^1$ or $Q_1 = \Sq^{1,2} + \Sq^{2,1}$. Together with the fact that a connected $A(1)$-module is free if and only if its Margolis homologies vanish (\cite{adamsmargolis1971modules}), setting $L = A(1)$ we see that $H^*(A(1) \otimes M; z) = 0 \otimes H^*(M;z) = 0$, from which the result follows.

As noted in \cite{stolz1992scpsc}, (2) and (3) follow (stably) from a classification result of Adams and Priddy in \cite{adamspriddy1976uniqueness}, and our isomorphism is the result of a direct computation. We will describe bases for the terms on both sides of (2) and (3). Also for these calculations, note that we have adopted the convention that $J$ and $K$ start in degree zero and $I$ starts in degree $1$.

Denote a ($\Z_2$) basis for $K$ by elements $k_l$ in degree $l$, so $K$ has basis $\{k_0, k_2, k_3\}$, and similarly denote the basis of $J$ by $\{j_0, j_1, j_2, j_3, j_4\}$. For $I$ we adopt the same convention, labeling elements by $i_l$s, but there are two basis elements in degree 3; let $i_3$ be the element such that $i_3 = \Sq^2 i_1$ and $i_3'$ the element with $i_3' = \Sq^1 i_2$. For tensors omit the tensor product symbol, so $j_l \otimes k_m \in J \otimes K$ will be written $j_l k_m$.

For the isomorphism (2), $j_2k_0$ generates a copy of $I$ under the action by $A(1)$ and $j_0k_0$ generates the $A(1)$. For (3), the copies of $A(1)$ are generated by $j_0i_1, j_0i_2, j_1i_2,$ and $j_2i_2,$ and the copy of $K$ is generated by $j_0i_3' + j_1i_2 + j_2i_1$. The full computation of this fact is given in Figure \ref{fig:piecestensor}.
\end{proof}

\begin{center}
\begin{figure}
    \includegraphics[scale=0.7]{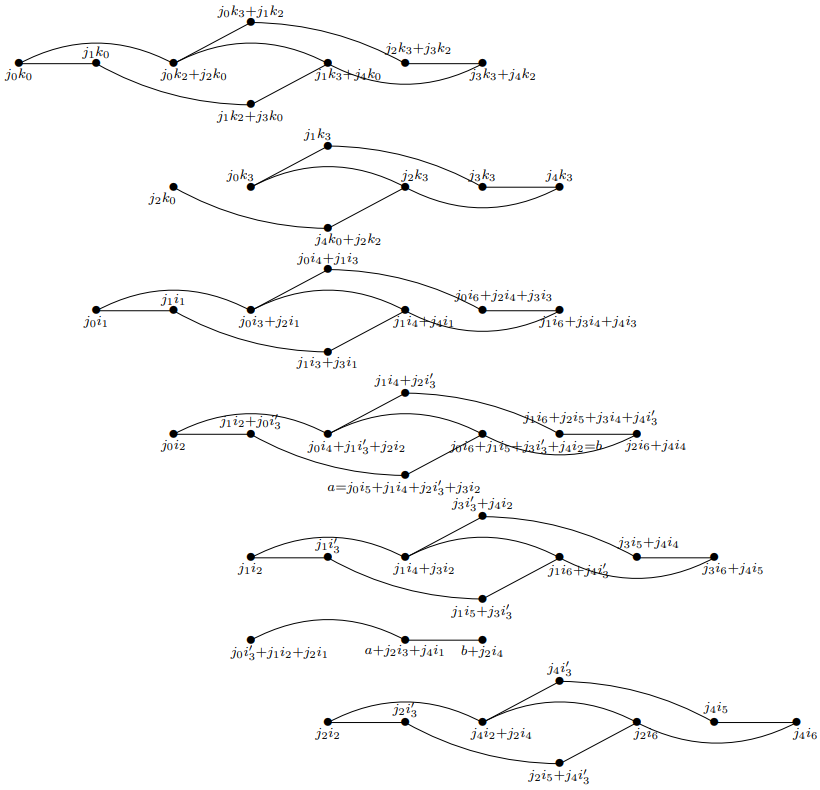}
    \caption{Explicit calculation for the proof of Lemma \ref{lem:piecestensor}, parts (2) and (3). Modules are aligned according to degree.}
    \label{fig:piecestensor}
\end{figure}
\end{center}

We have shown

\begin{theorem} \label{thm:mspinhstructure}
As an $A(1)$-module, $\overline{H^*\mathrm{MSpin^h}}$ is a direct sum of suspensions of $A(1), I,$ and $K$. Each of these summands can be located by locating the appropriate summands of $A(1), \Z_2 = A(1)/I$, and $J$ in $\overline{H^*\mathrm{MSpin}}$ and $A(1), I,$ and $K$ in $H^*MSO_3$ and then applying \ref{lem:piecestensor}.
\end{theorem}

To locate the summands for $H^*\mathrm{MSpin}$ one uses the Poincar\'e series found in \cite{abp1966spinshort}. For $H^*MSO_3$ the summands can also be located via Poincar\'e series, given in Appendix D of \cite{freedhopkins2021refl}. Hence the homotopy groups of $\mathrm{MSpin^h}$ are easily determined from this information; via \cite{stolz1994splitting}, for an $\mathrm{MSpin}$-module spectrum $Y$ summands of $\overline{H^*Y}$ correspond to wedge summands of $Y$ by the proof of Corollary \ref{cor:mspinhhtpy} below: $K$ corresponds to the spectrum $\mathrm{ksp} = \Sigma^{-4} \mathrm{ko} \langle 4 \rangle$, $A(1)$ the Eilenberg-MacLane spectrum $H\Z_2$, and $I$ the cofiber $C$ of the map $\mathrm{ko} \to H\Z_2$ representing the nontrivial class in $H^0\mathrm{ko}$. Let $\widetilde{\mathrm{ko}} = \Sigma^{-1}C$; it is easy to see that $\widetilde{\mathrm{ko}}$ has homotopy groups identical to those of $\mathrm{ko}$.

\begin{corollary} \label{cor:mspinhhtpy}
$\mathrm{MSpin^h}$ is 2-locally a generalized wedge sum of copies of $H\Z_2$, $\mathrm{ksp}$, and $\widetilde{\mathrm{ko}}$.
\end{corollary}
\begin{proof}
    Apply Corollary 4.2 of \cite{stolz1994splitting} several times to obtain this assertion as follows. First, there are finite spectra $X, X_I,$ and $X_K$ whose cohomologies are isomorphic as $A(1)$-modules to $A(1), \Sigma^{-1}I, $ and $K$ (resp.) (cf. Propositions 8.5 of \cite{stolz1994splitting} and 2.1 of \cite{davismahowald1981periodicity}). Then apply Corollary 4.2 of \cite{stolz1994splitting} first to $\mathrm{MSpin^h}$ to obtain a generalized wedge sum of $\mathrm{ko} \wedge (X, X_{I}, \text{or} X_{K})$, and then once each to $H\Z_2 \simeq_{2} \mathrm{ko} \wedge X$, $\widetilde{\mathrm{ko}} \simeq_{2} \mathrm{ko} \wedge X_{I}$, and $\mathrm{ksp} \simeq_{2} \mathrm{ko} \wedge K$; it is easy to see that the $A(1)$-modules present satisfy the assumptions for Stolz's splitting results since copies of $\Sigma^{-1}I$ and $K$ appear in degrees $0$ mod $4$, and $A(1)$ is free.
\end{proof}

\section{A Counting Procedure}\label{sec:counting}

In this section we carry out the ``instructions" in Theorem \ref{thm:mspinhstructure} for locating summands to describe a precise decomposition of $\mathrm{MSpin^h}$ in degrees $\leq 30$.

To begin let us start with the easier component $E = \Sigma^{-3}MSO_3$. Additively we have $H^*E \cong \Z_2[w_2U, w_3U]$ for the Thom class $U$. In \cite{freedhopkins2021refl}, Freed and Hopkins show that the $A(1)$-module structure of $H^*E$ is $H^*E \cong (K \oplus \Sigma^3 I)\otimes \Z_2[w_4^2] \oplus \text{free}$. Denote the Poincar\'e series for a graded $\Z_2$-vector space by $p_M(t)$, so that we have
\begin{align*}
    p_K(t) &= 1 + t^2 + t^3, \\
    p_{\Sigma^{-1}I}(t) &= 1 + t + 2t^2 + t^3 + t^4 + t^5, \\
    p_{ \Z_2[w_4^2]}(t) &= \frac{1}{1-t^8}, \\
    p_E(t) &= \frac{1}{(1-t^2)(1-t^3)}, \\
    p_{A(1)}(t) &= (1+t)(1+t^2)(1+t^3).
\end{align*}
Then one can locate the free modules by computing
\[
\frac{1}{p_{A(1)}(t)}\left(p_E(t) - p_{ \Z_2[w_4^2]}(t)(p_K(t) + p_{\Sigma^3 I}(t))\right),
\]
which gives $t^9 + t^{15} + t^{17} + t^{21} + t^{23} + t^{25} + t^{27} + t^{29} + O(t^{31})$, telling us that $H^*E$ has a copy of $A(1)$ in degrees 9, 15, 17, etc. We also know that $H^*E$ has a copy of $K$ in all degrees 0 mod 8 and $I' = \Sigma^{-1}I$ in degrees 4 mod 8. Thus in degrees $\leq 30$, we have the following decomposition of $H^*E$:
\begin{center}
    \begin{tabular}{|c||c|c|c|c|c|c|c|c|}
    \hline
    degree & 0 & 4 & 8 & 9 & 12 & 15 & 16 & 17 \\
    summand & $K$ & $I'$ & $K$ & $A(1)$ & $I'$ & $A(1)$ & $K$ & $A(1)$\\
    \hline
    degree & 20 & 21 & 23 & 24 & 25 & 27 & 28 & 29 \\
    summand & $I'$ & $A(1)$ & $A(1)$ & $K$ & $A(1)$ & $A(1)$ & $I'$ & $A(1)$ \\
    \hline
\end{tabular}
\end{center}
In the above table, a summand $M$ in a particular degree denotes that the lowest-degree element of $M$ lives in the designated degree (so we have omitted suspensions), so explicitly our table means $H^*E = K \oplus \Sigma^{4}I' \oplus \Sigma^{8} K \oplus \Sigma^9 A(1) \oplus \cdots$.
\begin{remark}\label{rmk:Ishifts}
    The use of $I'$ is justified here since while $I$ corresponds to $\Sigma \widetilde{\mathrm{ko}}$, the homotopy groups of $\Sigma \widetilde{\mathrm{ko}}$ begin in degree 1, so using $I'$ we can more easily read off homotopy groups without exposing ourselves to off-by-one errors.
\end{remark}

For the $\mathrm{MSpin}$ component, a counting procedure outlined in \cite{abp1966spinshort} tells us the summands in $\overline{H^*\mathrm{MSpin}}$. We record the computation for degrees $\leq 30$ here.

Let $X$ be the graded vector space with generators $x_J$ in dimension $4n$ for each partition\footnote{These partitions cannot include 1, and we also allow $n=0$.} $J$ of $n$ with $n$ even. Similarly, let $Y$ be the graded vector space with generators $y_J$ in dimension $4n-2$ for each partition $J$ of $n$ with $n$ odd. Let $L_1 = A/A(\Sq^1, \Sq^2)$ and $L_2 = A/A(\Sq^3).$ Then Theorem 1.3 of \cite{abp1966spinshort} says that there is a graded vector space $Z$ such that
\[
H^*\mathrm{MSpin} \cong (L_1 \otimes X) \oplus (L_2 \otimes Y) \oplus (A \otimes Z).
\]
We can locate the free $A$-summands of $H^*\mathrm{MSpin}$, and hence the free $A(1)$-summands of $\overline{H^*\mathrm{MSpin}}$, by computing the Poincar\'e polynomial
\[
p_Z(t) = \frac{1}{p_A(t)}\left(p_{H^*\mathrm{MSpin}}(t) - p_{L_1}(t)p_X(t) - p_{L_2}(t)p_Y(t)\right).
\]
These Poincar\'e polynomials, truncated sufficiently for our purposes of computing $p_Z(t)$ up through degree $30$, are
\begin{align*}
    p_{H^*\mathrm{MSpin}}(t) &= \prod_{\substack{n > 3 \\ n \neq 2^r + 1}}(1-t^n)^{-1} = \prod_{\substack{4 \leq n \leq 40 \\ n \neq 5, 9, 17, 33}}(1-t^n)^{-1}, \\
    p_{L_1}(t) &= \prod_{\substack{n = 2^r - 1 \\ r \geq 3}}(1-t^n)^{-1}(1-t^4)^{-1}(1-t^6)^{-1} = \prod_{n=4,6,7,15,31}(1-t^n)^{-1}, \\
    p_{L_2}(t) &= p_{L_1}(t)(1 + t + t^2 + t^3 + t^4), \\
    p_A(t) &= \prod_{\substack{n = 2^r - 1 \\ r \geq 1}}(1-t^n)^{-1} = \prod_{n=1,3,7,15,31}(1-t^n)^{-1}, \\
    p_X(t) &= 1 + t^8 + 2t^{16} + 4t^{24} + 7t^{32}, \\
    p_Y(t) &= t^{10} + 2t^{18} + 4t^{26} + 8t^{34}.
\end{align*}
Hence $p_Z(t) = t^{20} + t^{22} + 2t^{28} + t^{29} + 3t^{30} + O(t^{32}).$ From this, we can create a similar table to the above for $\overline{H^*\mathrm{MSpin}}$, which is
\begin{center}
    \begin{tabular}{|c||c|c|c|c|c|c|}
    \hline
        degree & 0 & 8 & 10 & 16 & 18 & 20\\
        summand & $\Z_2$ & $\Z_2$ & $J$ & $\Z_2^2$ & $J^2$ & $A(1)$\\
        \hline
        degree & 22 & 24 & 26 & 28 & 29 & 30\\
        summand & $A(1)$ & $\Z_2^4$ & $J^4$ & $A(1)^2$ & $A(1)$ & $A(1)^3$\\
    \hline
    \end{tabular}
\end{center}

Given this information we can compute tensor products and use Lemma \ref{lem:piecestensor} to decompose $H^*\mathrm{MSpin^h}$. Being careful to keep track of suspensions we obtain

\begin{center}
    \begin{tabular}{|c|c|c|c|c|c|}
        \hline
        0 & 4 & 8 & 9 & 10 & 12 \\
        $K$ & $I'$ & $K^2$ & $A(1)$ & $A(1)$ & $(I')^3$\\
        \hline
        14 & 15 & 16 & 17 & 18 & 19 \\
        $A(1)$ & $A(1)^2$ & $K^5 \oplus A(1)$ & $A(1)^3$ & $A(1)^3$ & $A(1)$ \\
        \hline
        20 & 21 & 22 & 23 & 24 & 25\\
        $(I')^7 \oplus A(1)^2$ & $A(1)^2$ & $A(1)^6$ & $A(1)^7$ & $K^{11} \oplus A(1)^5$ & $A(1)^{10}$ \\
        \hline
        26 & 27 & 28 & 29 & 30 & $\geq$31\\
        $A(1)^{11}$ & $A(1)^7$ & $(I')^{15} \oplus A(1)^{10}$ & $A(1)^{10}$ & $A(1)^{19}$ & $\cdots$ \\
        \hline
    \end{tabular}
\end{center}

\noindent Hence we have the 2-local splitting
\[
\mathrm{MSpin^h} \simeq \mathrm{ksp} \vee \Sigma^4\widetilde{\mathrm{ko}} \vee \Sigma^8 \mathrm{ksp} \vee \Sigma^8 \mathrm{ksp} \vee \Sigma^9 H\Z_2 \vee \cdots
\]

\section{Further Results} \label{sec:more}
Here we discuss the map $\Ahath: \mathrm{MSpin^h} \to \mathrm{ksp}$. A more complete exposition can be found in \cite{Hu}, but for our purposes it suffices to note that $\Ahath$ is a $\mathrm{MSpin}$-module spectrum map over the spin orientation $\hat{\mathcal{A}}: \mathrm{MSpin} \to \mathrm{ko}$. Hu proves that on homotopy, the induced map (which we will also call $\Ahath$) is epic. We can strengthen this result:

\begin{lemma} \label{lem:Ahath_splitepi}
$\Ahath: \mathrm{MSpin^h} \to \mathrm{ksp}$ induces a split surjection on homotopy groups.
\end{lemma}
\begin{proof}
Let $F \to \mathrm{MSpin^h} \xrightarrow{\Ahath} \mathrm{ksp}$ be the corresponding fiber sequence. Since $\Ahath$ is surjective on homotopy, there are short exact sequences
\[
0 \to \pi_k F \to \pi_k \mathrm{MSpin^h} \to \pi_k \mathrm{ksp} \to 0
\]
for all $k$.

When $k \neq 5,6$ mod 8, $\pi_k \mathrm{ksp}$ is either 0 or $\Z$, hence is projective, so the sequences split in this case. When $k$ is 5 or 6 mod 8, we have $\pi_k = \Z_2$. We will construct a right splitting $r: \pi_k \mathrm{ksp} \to \pi_k \mathrm{MSpin^h}$ in these cases.

Let $M_5 = \HP^1_+ \times S^1_\eta$, where $\HP^1_+$ is $\HP^1 \cong S^4$ with the \spinh structure described in 2.29 of \cite{Hu} and $S^1_\eta$ is $S^1$ with its non-bounding spin structure. Then $M_5$ is a \spinh 5-manifold. Now, as computed by Hu, $\Ahath$ sends $\HP^1_+$ to $1$ on $\pi_4$, so $\Ahath$ must send $M_5$ to the generator $\eta$ of $\pi_5 \mathrm{ksp}$, otherwise we contradict the fact that multiplication by $\eta$ is a surjection $\Z = \pi_4 \mathrm{ksp} \to \pi_5 \mathrm{ksp} = \Z_2$. Similarly, setting $M_6 = M_5 \times S^1_\eta$ we see that $\Ahath(M_6) \neq 0$; that is, $\Ahath(M_6)$ generates $\pi_6 \mathrm{ksp}$.

For $k=5,6$ define $r(u) = [M_k]$ where $u$ denotes a generator of $\pi_k \mathrm{ksp}$. Then $\Ahath \circ r (u) = \Ahath[M_k] = u$, so $\Ahath \circ r = 1_{\pi_k \mathrm{ksp}}$ and $r$ is indeed a right splitting.

The above forms a base case for an induction argument for the general case, which proceeds as follows.

Let $k \equiv 5,6$ mod 8 and let $M_k$ be a $k$-dimensional \spinh manifold with $\Ahath[M_k] \neq 0$. Then take $B$, a Bott manifold (a simply connected spin $8$-manifold with $\hat{A}$-genus equal to 1). Since $\Ahath$ is an $\mathrm{MSpin}$-module spectrum map over $\hat{\mathcal{A}}$, we have that $\Ahath[B \times M_k] = \hat{\mathcal{A}}([B])\Ahath([M_k]) = \beta \cdot u \neq 0$, where $\beta$ is the Bott element in $\pi_8 \mathrm{ko}$ and $u$ generates $\pi_k \mathrm{ksp}$.

Then, defining $r: \pi_{k+8} \mathrm{ksp} \to \pi_{k+8} \mathrm{MSpin^h}$ via $\beta \cdot u \mapsto [B \times M_i]$ gives a right splitting of the short exact sequence in degreee $k+8$. Thus the sequence splits for all $k \equiv 5,6$ mod $8$.
\end{proof}

With the decomposition provided in section \ref{sec:counting} one also sees that $\Ahath$ is the projection onto the bottom summand of $\mathrm{MSpin^h}$.

\bibliography{spinh}
\bibliographystyle{alpha}

\end{document}